\documentclass[12pt, leqno]{amsart}
\usepackage[all]{xy}
\usepackage{amsfonts,amsmath,oldgerm,amssymb,amscd}
\UseComputerModernTips
\usepackage[utf8]{inputenc} 
\usepackage[T1]{fontenc}
\usepackage{tabularx,ragged2e,booktabs,caption}

\newtheorem{theorem}{Theorem}
\newtheorem{lemma}[subsection]{{\bf Lemma}}
\newtheorem{coro}[subsection]{{\bf Corollary}}

\newcommand{\del}{\delta}

\newcommand{\Z}{\mbox{$\mathbb Z$}}
\newcommand{\Q}{\mbox{$\mathbb Q$}}

\newcommand{\ps}{$\psi_n^{(\alpha)}(x)$}
\newcommand{\g}{g_n^{(\alpha)}(x)}
\newcommand{\G}{G_n^{(\alpha)}(x)}
\newcommand\numberthis{\addtocounter{equation}{1}\tag{\theequation}}

\begin{document}
\title{Irreducibility of extensions of Laguerre Polynomials } 
\author[Laishram, Nair]{Shanta Laishram, Saranya G. Nair}
\address{Stat-Math Unit, India Statistical Institute\\
7, S. J. S. Sansanwal Marg, New Delhi, 110016, India}
\email{shanta@isid.ac.in}
\author[Shorey]{T. N. Shorey}
\address{Department of Mathematics\\
IIT Bombay, Powai, Mumbai 400076, India}
\email{saranya@math.iitb.ac.in}
\email{shorey@math.iitb.ac.in}
\thanks{2000 Mathematics Subject Classification: Primary 11A41, 11B25, 11N05, 11N13, 11C08, 11Z05.\\
Keywords: Irreducibility,  Laguerre Polynomials, Primes, Newton Polygons.}
\begin{abstract}
For integers $a_0,a_1,\ldots,a_n$ with $|a_0a_n|=1$ and either $\alpha =u$
with $1\leq u \leq 50$ or $\alpha=u+ \frac{1}{2}$ with $1 \leq u \leq 45$, we prove that
$\psi_n^{(\alpha)}(x;a_0,a_1,\cdots,a_n)$ is irreducible except for an explicit
finite set of pairs $(u,n)$. Furthermore all the exceptions other than $n=2^{12},\alpha=89/2$ are
necessary. The above result with $0\leq\alpha \leq 10$ is due to Filaseta, Finch
and Leidy and
with $\alpha \in \{-1/2,1/2\}$ due to Schur.
\end{abstract}
\maketitle
\pagenumbering{arabic}
\pagestyle{myheadings}
\markright{Irreducibility of extensions of Laguerre polynomials}
\markleft{Laishram, Nair and Shorey}

\section{Introduction}
For positive integer $n$ and real number $\alpha$, the generalized Laguerre polynomial is given by
 \begin{align*}
L_n^{(\alpha)}(x)=\displaystyle\sum_{j=0}^{n}\frac{(n+\alpha)(n-1+\alpha)\dots (j+1+\alpha)}{(n-j)!j!}(-x)^j
\end{align*}
and $L_n^{(0)}(x)$ is called Laguerre polynomial.  We shall restrict ourselves to the case when $\alpha$ is a rational number  written uniquely as 
  \begin{align}\label{alpha1}
  \alpha=\alpha(u)=u+\frac{a}{d}
  \end{align}
  where $u,a , d \in \mathbb{Z}$ and $d \geq 1$ such that  $a=0$ if $d=1$ and $1 \leq a <d,\ \gcd(a,d)=1$ if $d >1.$ Thus $\alpha=u$ if $d=1.$ The generalized Laguerre polynomial satisfies second order linear differential equation
\begin{align*}
x y^{''}+ (\alpha+1-x)y^{'}+ny=0,\ y=L_n^{(\alpha)}(x)
\end{align*}
and the difference equation
\begin{align*}
L_n^{(\alpha)}(x)-L_n^{(\alpha-1)}(x)=L_{n-1}^{(\alpha)}(x).
\end{align*} They have been studied in various branches of mathematics and mathematical physics and there is an extensive literature on them, see \cite{Sz75}. Schur \cite{Sch31}, \cite{Sch1} was the first to establish interesting and important algebraic properties of these polynomials. In particular, the irreducibility of these polynomials has been well investigated, see \cite{LaNaSh15} for an account of results in this connection. 
 Filaseta, Finch and Leidy \cite{Fil Influ} showed that  $L_n^{(\alpha)}(x)$ is irreducible for all $n$ and integers $\alpha$ with $0 \leq \alpha \leq 10$ except when $(n,\alpha) \in \{(2,2),(4,5),(2,7)\}.$ Laishram and Shorey \cite{LaSh11} extended it for integers $\alpha$ with $0 \leq \alpha \leq 50$ and showed that $L_n^{(\alpha)}(x)$ is irreducible for all $n$ except for $n=2, \alpha \in \{2,7,14,23,34,47\}$ and $n=4, \alpha \in \{5,23\}$ where it has a linear factor. Further explicit factorizations for these exceptions have been given.
In this paper, we consider more general polynomials than $L_n^{(\alpha)}(x)$. By irreducibility of a polynomial, we shall always mean that it is irreducible over $\mathbb{Q}.$ Observe that if a polynomial of degree $m$ has a factor of degree $k,$ then it has also a factor of degree $m-k.$ {\it Therefore given a polynomial of degree $m$, we always consider a factor of degree $k$ where $1 \leq k \leq \frac{m}{2}$}. Let $a_0,a_1, \ldots, a_n$ be integers with $|a_0|=|a_n|=1$ and $\alpha$ be a rational number given by \eqref{alpha1}. Then we define 
\begin{align*}
\psi_n^{(\alpha)}(x)= \psi_n^{(\alpha)}(x;a_0,a_1,\ldots,a_n)&= \displaystyle\sum_{j=0}^{n}a_j \binom{n}{j}(n+\alpha)\cdots (j+1+\alpha)d^{n-j}x^j\\
& =\displaystyle\sum_{j=0}^{n}a_j \binom{n}{j}(a+(u+n)d)\cdots (a+(u+j+1)d)x^j \numberthis \label{ps}
\end{align*}
by \eqref{alpha1}. We observe that 
\begin{align*}
 \psi_n^{(\alpha)}(x) = d^nn!L_n^{(\alpha)}\left(\frac{x}{d}\right)\ \text {if}\ a_j=(-1)^j
\end{align*}
and therefore, \ps \ with $a_j=(-1)^j$ is irreducible if and only if $L_n^{(\alpha)}(x)$ is irreducible. Thus the irreducibility of \ps \ implies the irreducibility of $L_n^{(\alpha)}(x).$ Filaseta, Finch and Leidy \cite{Fil Influ} showed that \ps \ with $0 \leq \alpha \leq 10$ and $|a_0a_n|=1$ is irreducible for all $n$ except $(n, \alpha) \in \{(2,2),(2,7),(4,4),(4,5),(8,8),(24,8)\}$ where \ps \ has a linear factor. They proved that these exceptions are necessary in the sense that there exist integers $a_0,a_1,\ldots, a_n$ with $|a_0|=|a_n|=1$ such that \ps \ has a linear factor. We extend their results for $\alpha \leq 50.$ Let 
\begin{align*}
\Omega=&\{(2,14), (2,23), (2,34),(2,47),(4,14), (4,20), (4,23), (6,44), (8,41), (12,24),\\&(16,16), (16,20), (16,24),(16,29), (24,24), (30,24),(32,32),(32,48), (40,24),\\& (48,24), (112,48),(120,24)\}.
\end{align*}
\begin{theorem}\label{Thm1}
Let $11 \leq \alpha \leq 50$ be an integer and $|a_0a_n|=1$. Then \ps \ is irreducible except when $(n,\alpha) \in \Omega$ where it may have a linear factor or $(n,\alpha)=(16,24)$ where it may have a quadratic factor. Further for every $(n,\alpha) \in \Omega$,  there exist integers $a_0,a_1, \ldots,a_n$ with $|a_0|=|a_n|=1$ such that \ps \ has a linear factor and further integers $a_0,a_1, \ldots,a_n$ with $|a_0|=|a_n|=1$ such that \ps \ has a quadratic factor for 
$(n,\alpha)=(16,24)$.  The  factors for each case are given in the following table.
\begin{center}
\captionof{table}{} 
  \begin{tabular}{|c|c|}
    \hline
    $x\pm b$ & $(n,\alpha)$ \\
    \hline
      $x\pm 2$  & $(16,16),(32,32)$  \\
    \hline
    $x\pm 4$  & $(2,14)$  \\
       \hline 
       $x\pm 6$  & $(2,34),(4,14),(4,20),(4,23),$\\ & $(12,24),(16,20),(24,24),(48,24)$ \\
        \hline
                  $x\pm 10$  & $(32,48)$  \\
       \hline
           $x\pm 14$  & $(8,41)$  \\
    \hline
    $x\pm 20$  & $(2,23)$  \\
     \hline
        $x\pm 30$  & $(6,44),(16,29),(30,24),(40,24),(120,24)$  \\
      \hline
     $x\pm 56$  & $(2,47)$ \\
            \hline
       $x\pm 70$ & $(112,48)$\\
        \hline
        $x\pm 150$, $x^2\pm 780$ & $(16,24)$\\
              \hline    
    \end{tabular}
    \end{center}
  \end{theorem}
Thus the exceptions in Theorem \ref{Thm1} are necessary.

Next we consider $L_n^{(\alpha)}(x)$ and more generally $\psi_n^{(\alpha)}(x)$ when $\alpha$ is a rational number with denominator $2.$ Thus $\alpha=u+\frac{1}{2}$
 where $u$ is an integer. We recall that Hermite polynomials $H_{2n}(x)$ and $H_{2n+1}(x)$ are given by $$H_{2n}(x)=(-1)^n2^{2n}n!L_{n}^{(-\frac{1}{2})}(x^2) \ \text{and} \ H_{2n+1}(x)=(-1)^n2^{2n+1}n!x L_{n}^{(\frac{1}{2})}(x^2).$$ 
  Schur \cite{Sch31}, \cite{Sch1} proved that $L_n^{(-\frac{1}{2})}(x^2)$ and $L_n^{(\frac{1}{2})}(x^2)$ are irreducible and this implies the irreducibility of $H_{2n}(x)$ and $H_{2n+1}(x)/x.$ We observe that $u \in \{-1,0\}$  in these results. Further, Laishram, Nair and Shorey \cite{LaNaSh15} showed that $L_n^{(\alpha)}(x^2)$ with $\alpha=u+\frac{1}{2}$
 and $1 \leq u \leq 45$  are irreducible except when $(u,n)=(10,3).$ In such a case  $L_3^{(\frac{21}{2})}(x^2)=\frac{-1}{48}(2x^2-15)(4x^4-132x^2+1035)$. It follows immediately from the above results that $L_n^{(\alpha)}(x)$ with $\alpha=u+\frac{1}{2}$ and $-1 \leq u \leq 45$  are irreducible except when $(u,n)=(10,3)$ where it has a linear factor. Further in the next result, we compute the Galois group of $L_n^{(\alpha)}(x)$ when $\alpha=u+
 \frac{1}{2}.$ We prove 
 
 \begin{theorem}\label{GGLu}
 Let  $\alpha=u+
  \frac{1}{2}$ and $-1 \leq u\leq 45$. Then the Galois group of $L^{(\alpha)}_n(x)$ 
 is $S_n$ except when $(u,n)=(10,3)$ in which case the Galois group is $\mathbb{Z}_2$.
 \end{theorem}
 Laishram \cite{la15hardy} proved that the Galois group is $S_n$ when $u \in \{-1,0\}.$ Therefore we assume that $u \geq 1$ in the proof of Theorem \ref{GGLu}. By putting $a=1,d=2$ in \eqref{ps}, we have 
 \begin{align*}
 \psi_n^{(\alpha)}(x)=\displaystyle\sum_{j=0}^{n}a_j \binom{n}{j}(1+2(u+n))\cdots (1+2(u+j+1))x^j.
 \end{align*} 
 It follows from the results of Schur on $G_n^{(\alpha)}(x)$ stated in the next section before Lemma \ref{ psi greater than 2 with half} that  $\psi_n^{(\alpha)}(x^2)$ with $|a_0|=|a_n|=1$ is irreducible when $u =-1$ and also $u=0$ unless $2n+1$ is a power of $3$ where it may have a linear factor or quadratic factor. Let 
 \begin{align*}
 \Omega_1=&\{(2,2),(2,8),(2,2^9),(6,2^4),(9,4),(9,2^6),(10,3),(10,12),
 (10,24), (10,192),(16,8) \\
 &(21,2^4), (24,2^4),(30,2^6),(35,2^5),(35,2^{9}),(37,12),(37,36),(37,144),(44,2^{12})\}.
 \end{align*}
  \begin{theorem}\label{Thm2}
  Let $\alpha=u+\frac{1}{2}$
 where $1 \leq u \leq 45$  is an integer. Then $\psi_n^{(\alpha)}(x^2)$ with $|a_0a_n|=1$ is 
 irreducible except when $(u,n) \in \Omega_1$ where it may have a 
 quadratic factor or $(u,n)=(9,4)$ where it may have a factor of degree $4$. Further for every $(u,n) \in \Omega_1$ except for $(u,n)=(44,2^{12}),$ there exist integers $a_0,a_1, \ldots,a_n$ with $|a_0|=|a_n|=1$ such that $\psi_n^{(\alpha)}(x^2)$ has a quadratic factor. The quadratic factors are given in the following table.
 \begin{center}
 \captionof{table}{} 
\begin{tabular}{|c|c|}
 \hline
  $x^2\pm b$ & $(u,n)$ \\
 \hline
 $x^2\pm 3$ & $(9,4),(10,3),(24,2^4)$ \\ 
 \hline
  $x^2\pm 15$ & $(6,2^4),(10,12),(10,192),(21,2^4),(35,2^5)$ \\
 \hline
  $x^2 \pm 21$ & $(2,2),(2,8),(2,2^9),(9,2^6),(30,2^6),(37,36)$ \\
 \hline
  $x^2\pm 33$ & $(37,12),(37,144)$ \\
 \hline
 $x^2\pm 69$ & $(10,24)$  \\
 \hline
  $x^2\pm 1095$ & $(35,2^9)$  \\
 \hline
 $x^2\pm 7$ & $(16,8)$\\
 \hline
  \end{tabular}
  \end{center}
  \end{theorem}
We have not been able to find a factorization for $(u,n)=(44,2^{12})$ since $n$ is very large. We observe that the irreducibility of $\psi_n^{(\alpha)}(x^2)$ implies the irreducibility of \ps. 
  \begin{coro}\label{coro 2}
   Let $\alpha=u+\frac{1}{2}$
   where $1 \leq u \leq 45$  is an integer. Then $\psi_n^{(\alpha)}(x)$ with $|a_0a_n|=1$ is irreducible except when $(u,n) \in \Omega_1$  where it may have a linear factor or $(u,n)=(9,4)$ where it may have a quadratic factor. Further for every $(u,n) \in \Omega_1$ except for $(u,n)=(44,2^{12}),$ there exist integers $a_0,a_1, \ldots,a_n$ with $|a_0|=|a_n|=1$  such that $\psi_n^{(\alpha)}(x)$ has a linear factor. The linear factors are obtained from the above table with $x^2$ replaced by $x.$
  \end{coro}

  \section{Preliminaries}
  
  The proofs of Theorems \ref{Thm1}
 and \ref{Thm2}  depend on Newton polygons which we introduce now. Let $f(x)=\displaystyle\sum_{j=0}^{n}a_jx^j \in \mathbb{Z}[x]$ with $a_0a_m \neq 0$  and $p$ be a prime number. Let $S$ be the set of points in the extended plane
  $$S =\{(0,\nu(a_m)),(1,\nu(a_{m-1})),(2,\nu(a_{m-2})),\ldots,(m,\nu(a_0))\}$$ where for an integer $r,$ we write $\nu(r)=\nu_p(r)$ for the highest power of $p$ dividing $r$ and we put $\nu(0)=\infty.$ Consider the lower edges along the convex hull of these points. The left most endpoint
 is $(0,\nu(a_m))$ and the right most endpoint is $(m,\nu(a_0))$. The endpoints of each edge
 belong to S and the slopes of the edges increase from left to right. When referring
 to the edges of a Newton polygon, we shall not allow two different edges to have the
 same slope. The polygonal path formed by these edges is called the Newton polygon
 of $ f(x)$ with respect to the prime $p$ and we denote it by $NP_p(f)$. The endpoints of the
 edges on $NP_p(f)$ are called the vertices of $ NP_p(f)$. By a lattice point on an edge, we mean a
 lattice point on the edge
 other than the vertices of the edge. We denote the Newton function of
 $f$ with respect to the prime $p$ as the real function $f_p(x)$ on the interval $ [0,m]$ which
 has the polygonal path formed by these edges as its graph. Hence $f_p(i) = \nu(a_{m-i})$
 for $i = 0,$ $ m$ and at all points $i$ such that $(i,\nu(a_{m-i}))$ is a vertex of $NP_p(f)$. We need the following result of Dumas \cite{Dum06}.
 \begin{lemma}\label{Dumas}
 Let $g(x)$ and $h(x)$ be in $\mathbb{Z}[x]$ with $g(0)h(0)\neq 0$ and let $p$ be a prime. Let $k$ be a non-negative integer such that $p^k$ divides the leading coefficient of $g(x)h(x)$ but $p^{k+1}$ does not. Then the edges of the Newton polygon for $g(x)h(x)$ with respect to $p$ can be formed by constructing a polygonal path beginning at $(0,k)$ and using translates of the edges in the Newton polygons for $g(x)$ and $h(x)$ with respect to the prime p, using exactly one translate for each edge of the Newton polygons for $g(x)$ and $h(x)$. Necessarily, the translated edges are translated in such a way as to form a polygonal path with the slopes of the edges increasing.
 \end{lemma}
 Now we state a lemma of Filaseta \cite{Fil94} which is derived from Lemma \ref{Dumas}.
 \begin{lemma}\label{fil}
  Let $l, k,m$ be integers with $m \geq 2k > 2l \geq 0$. Suppose $g(x) =\displaystyle\sum_{j=0}^{m}b_jx^j \in
 \Z[x] $ and $p$ be a prime such that $ p \nmid b_m$ and $p\mid b_j $ for $0 \leq j \leq m-l-1$ and the
 right most edge of the $ NP_p(g)$ has slope $ <\frac{1}{k}$. Then for any integers $a_0, a_1,\ldots, a_m$
 with $p\nmid a_0a_m$, the polynomial $f(x) =\displaystyle\sum_{j=0}^{m} a_jb_jx^j$ cannot have a factor with degree in $[l + 1, k]$.
 \end{lemma}
 Next we state some earlier results on polynomials  which are more general than \ps. When $\alpha$ is an integer, the polynomials \ps \ are a special case of the following class of polynomials first considered
 by Schur. Let $n\ge 1, a\ge 0$ and $a_0, a_1, \ldots, a_n$ be integers. The \emph{generalized Schur polynomials} are
 defined as
 \begin{align}\label{f(x)}
 f_n^{(\alpha)}(x):=f_n^{(\alpha)}(x;a_0, a_1, \cdots, a_n)=\sum^n_{j=0} a_j\frac{x^j}{(j+\alpha)!}.
 \end{align}
We observe that $ (n+\alpha)!f_n^{(\alpha)}(x)=$ \ps \ if $a_j$ are replaced by $a_j\binom{n}{j}$ in the definition of $ f_n^{(\alpha)}(x).$
 
 Schur \cite{Sch1}, \cite{Sch31} proved that $ f_n^{(\alpha)}(x)$ with $\alpha \in \{0,1\}$ and $|a_0a_n|=1$ is irreducible unless $\alpha=1$ and $n+1=2^r$ for some $r$ where it may have a linear factor or $n=8$ where it may have a quadratic
 factor. Also for $\alpha=2$ and many other values of $\alpha$ the polynomial $ f_n^{(\alpha)}(x)$ may have a linear factor. Laishram and Shorey \cite{LaSh11} proved that
 \begin{lemma}\label{from extensions}
 Let $2 \leq k \leq \frac{n}{2}$
 and $a_0,a_1,\dots a_n \in \mathbb{Z}$  with $|a_0a_n|=1.$ Asssume that $0 \leq \alpha \leq 40$\ if $k=2$ and $0 \leq \alpha \leq 50$ if $k >2.$ Then $ f_n^{(\alpha)}(x)$ has no factor of degree $k$ except possibly when $(n,k,\alpha)$ is given by 
 \begin{align*}
 k=3,\ & (n,\alpha) \in \{(7,3),(8,2),(12,4),(46,4),(14,12),(17,11),(53,12)\}\\
 k=4,\ & (n, \alpha) \in \{(18,9),(18,10),(56,10),(16,12),(17,11),(38,13),(39,18)\}\\
 k=5,\ & (n,\alpha) \in \{(17,11),(19,9),(40,12)\}
 \end{align*}
  and $k=2$ with $(n, \alpha)$ satisfying
 \begin{itemize}
 \item[(i)]  $n+\alpha \leq 100$\\
 \item[(ii)] $\alpha \in \{13,14,19,33\}, \ n+\alpha \in \{126,225,2401,4375\}$\\
 \item[(iii)] $(n,\alpha) \in \{(112,9),(233,10),(234,9)\}$
  \end{itemize}
  together with the following set of pairs $(n,\alpha)$ given by the table:
  \begin{center}
   \captionof{table}{} 
  \begin{tabular}{|c|c||c|c||c|c|} \hline
  $\alpha$ & $n+\alpha$ & $\alpha$ & $n+\alpha$ & $\alpha$ & $n+\alpha$ \\ \hline
  $12$ & $169, 729$ & $15, 16$ & $289$ & $17$ & $513$ \\ \hline
  $18$ & $361, 513,  1216$ & $19, 20$ & $243$ & $21$ & $529$ \\ \hline
  $21, 22$ & $121, 576$ & $24$ & $325, 625, 676$ & $27$ & $784$ \\ \hline
  $28$ & $145$ & $29$ & $961$ & $31$ & $243$ \\ \hline
  $32$ & $243, 289, 1089$ & $33$ & $136, 256, 289, 5832$ & $36$ & $1369$ \\ \hline
  $38$ & $325, 625, 676$ & $39$ & $1025, 6561$ & $40$ & $288$ \\ \hline
  \end{tabular}
  \end{center}
    \end{lemma}
  The above result on $ f_n^{(\alpha)}(x)$ has a large number of exceptions especially when $k=2.$ Moreover it gives no information on linear factors. Hence in this paper, we consider \ps \ which is a special case of $f_n^{(\alpha)}(x)$, but more general than $L_n^{(\alpha)}(x)$ and we get complete irreducibility results for $ \psi_n^{(\alpha)}(x)$. Analogously we consider the  polynomial $G_n^{(\alpha)}(x)$ which is more general than \ps.
    For integers $a_0,a_1, \ldots,a_n$ and $\alpha$ given by \eqref{alpha1}, let 
      \begin{align*}
      G_n^{(\alpha)}(x)=G_n^{(\alpha)}(x;a_0,a_1, \ldots,a_n)=&\displaystyle\sum\limits_{j=0}^{n}a_{j}(n+\alpha)(n-1+\alpha)\cdots(j+1+\alpha)d^{n-j}x^j\end{align*}
      We observe that 
      \begin{align*}
     (n+\alpha)!f_n^{(\alpha)}(x)=\G \ {\rm  when} \ \alpha \ {\rm is \ an \ integer}.
      \end{align*} Let $\alpha$ be a rational with denominator $2.$ Then by \eqref{alpha1}, $\alpha=u+\frac{1}{2}$ and
      \begin{align*}
      \G = \displaystyle\sum\limits_{j=0}^{n}a_{j}x^j (\displaystyle\prod_{i=j+1}^{n}(1+2(u+i))).
      \end{align*}
  Schur \cite{Sch1}, \cite{Sch31} proved that  $G_n^{(\alpha)}(x^2)$ with $|a_0|=|a_n|=1$ is irreducible when $u \in \{-1,0\}$ unless $u=0$ and $2n+1$ is a power of $3$ where it may have a linear or quadratic factor. Let 
        $ A= \{\pm 2^t: t\geq 0,t \in \mathbb{Z}\}$ and 
      $S=\{(1,121),(8,59),(8,114),(9,4),(9,113),\\(9,163),(9,554),(15,23),(15,107),(16,106),(20,102),(21,101),(26,155),(26,287),(30,92),\\(36,86),(43,1158),(44,716)\}$. Laishram, Nair and Shorey \cite{LaNaSh15} proved the following irreducibility results on $G_n^{(\alpha)}(x^2).$  We observe that in \cite{LaNaSh15} the polynomials $G_n^{(\alpha)}(x)$ are denoted by $G_{\alpha}(x).$
    \begin{lemma}\label{ psi greater than 2 with half}
    Let $1\leq u \leq 45$ and $\alpha=u+\frac{1}{2}$. Let $a_0,a_n \in A.$ Then $G_n^{(\alpha)}(x^2)$ has no factor of degree $\geq 3$ except where 
       $(u, n)\in \{(1,12),(6,7),(9, 113),(10,3),(21, 101)\}$ or $(u,n) \in S$ or $(u,n)=(44,79)$ where it may have a factor of degree $3$ or $4$ or $6,$ respectively.
    \end{lemma}
     The proof for the irreducibility of $L_n^{(\alpha)}(x^2)$ given in Section 7 of \cite{LaNaSh15} based on Newton polygons is also valid for $\psi_n^{(\alpha)}(x^2)$  when $\alpha=u+\frac{1}{2}$ except for the pairs $(u,n) \in T_0$ where it may have a linear or quadratic factor where
     \begin{align*}
T_0=&\{(2,2),(2,8),(2,2^9),(6,2^4),(9,4),(9,2^6),(10,3),(10,12),(10,24),(10,192),(11,2),\\ &(16,2^3),(21,2^4),(24,2^4),(30,2^6),(35,2), 
(35,2^5),(35,2^{9}),(36,2^6), (37,12),(37,36),\\ &(37,144),(38,2), (44,2^{12})\}.
\end{align*} 
For $(u,n) \in T_0,$ we have computed $L_n^{(\alpha)}(x^2)$ in \cite{LaNaSh15} to find that it is irreducible except at $(u,n)=(10,3).$ 
But in this case of $\psi_n^{(\alpha)}(x^2)$, since $a_j$'s are arbitrary, we cannot exclude these pairs as we did it for 
$L_n^{(\alpha)}(x^2).$ Thus we have    
    
    \begin{lemma}\label{psi linear half}
    Let $1\leq u \leq 45$ and $\alpha=u+\frac{1}{2}$. Then $\psi_n^{(\alpha)}(x^2)$ with $|a_0a_n|=1$ has no factor of degree in $\{1,2\}$ except when $(u,n) \in T_0.$
    \end{lemma}
 \begin{lemma}\label{weger}
 The diophantine equation 
 \begin{align*}
 x+y=z
 \end{align*}
 in $x,y,z \in S=\{2^{x_{1}}\cdots 13^{x_{6}};x_i \in \mathbb{Z},x_i \geq 0\}$ with $\gcd (x,y)=1$ and $x \leq y$ has exactly $545$ solutions. Out of them $514$ satisfy 
 {\text ord}$_{2}(xyz) \leq 12,$ ord$_{3}(xyz) \leq 7$, ord$_{5}(xyz) \leq 5$, ord$_{7}(xyz) \leq 4,$ ord$_{11}(xyz) \leq 3,$ ord$_{13}(xyz) \leq 3.$ The remaining $31$ solutions are given in \cite[Table IX]{Weg}.
 \end{lemma} 
  This is due to de Weger \cite{Weg}. Further we need the following result from 
  \cite[Lemma 4.1]{Fil Influ} which is a direct application of Lemma \ref{Dumas} for determining $a_0,a_1,\ldots,a_n$ such that \ps \ has a linear factor when $(n,\alpha) \in \Omega.$
  \begin{lemma}\label{monic}
  Let $w(x)$ be a monic polynomial in $\mathbb{Z}[x]$ divisible by $x-b$ with $b \in \mathbb{Z}.$ Let $p$ be a prime and $e$ be a non- negative integer for which $p^e\parallel b.$ Then $NP_p(w(x))$ with respect to $p$ has an edge that includes a translate of the line segment joining $(0,0)$ to $(1,e).$ Also, if the right most edge has slope $<1,$ then necessarily $e=0.$
  \end{lemma} 
  \begin{lemma}\label{order}
  Let $p$ be a prime. For any integer $l\geq 1$, write $l$ in base $p$ as $l=l_tp^t+l_{t-1}p^{t-1}+\dots+l_1p+l_0$ where $0\leq l_i\leq p-1$ for $0\leq i \leq t$ and $l_t >0$. Then
  \begin{align*}
  \nu_p(l!)=\frac{l-\sigma_p(l)}{p-1}
  \end{align*}
  where $\sigma_p(l)=l_t+l_{t-1}+\dots+l_1+l_0$.
  \end{lemma}
  This is due to Legendre. For a proof, see \cite[Ch.17, p 263]{Hasse}.
  \begin{lemma}\label{prime cong}
  Let $r \in \{1,3\}.$ The interval $(x,1.048x]$ contain  primes congruent to $r$  modulo $4$ when 
  $x \geq 887.$ 
  \end{lemma}
  This follows from \cite[Theorem 1]{CuHa} with $k=4.$
  
 \section{Lemmas for the proof of Theorem \ref{Thm1}} 
Let $\alpha$ be an integer throughout this section. We write
  \begin{align*}
  \Delta_j=\Delta(\alpha+1, j)=(\alpha+1)(\alpha+2)\cdots (\alpha+j).
  \end{align*} For the proof of Theorem \ref{Thm1}, we need the following result which is an analogous for $f_n^{(\alpha)}(x)$ as proved in \cite[Lemma 1.1]{LaSh11}.
 
 \begin{lemma}\label{Glemma}
  Let $\alpha>0, 1\le k\le \frac{n}{2}$ and $u_0=\frac{\alpha}{k}$. Assume that there is a prime $p\ge k+2$ with
 \begin{align}\label{cond}
 p|\prod^k_{i=1}(n-k+i)(\alpha+n-k+i), \ \ p\nmid a_0a_n\prod^k_{i=1}(\alpha+i).
 \end{align}
 Suppose
 \begin{align}\label{2u0}
 p\ge \min(2u_0, k+u_0)
 \end{align}
 or
 \begin{align}\label{2k}
 p>2k \ {\rm and} \ p^2-p\ge \alpha.
 \end{align}
 Then \ps \ has no factor of degree $k$.
\end{lemma}
The proof of Lemma \ref{Glemma} is exactly same as the proof of \cite[Lemma 1.1]{LaSh11} for $f_n^{(\alpha)}(x)$. Further we prove the following result analogus to Lemma \ref{Glemma} with $k \in \{1,2\}.$
\begin{lemma}\label{Glemma1}
  Let $k\in \{1, 2\}$ and $p\ge 2k+1$ be such that $p|\prod^k_{i=1}(n-k+i)$, $0 < \alpha \leq 50$ and
 \begin{align*}
 \nu_p((\alpha+1)\cdots (\alpha+j))\le \nu_p(n(n-1)\cdots (n-j+1)) \ {\rm for} \ 1\le j\le k.
 \end{align*}
 Then \ps \ with $|a_0a_n|=1$ has no factor of degree $k$ except when $k=1, p=3, \alpha \in \{24,25\}$ and
 $\nu_3(n)=1$.
 \end{lemma}
 \begin{proof}
We use Lemma \ref{fil} with $g(x)= \g, m=n,$ and $ l=k-1$ where
\begin{align}\label{g}
\g=\displaystyle\sum_{j=0}^{n}\binom{n}{j}(n+\alpha)\cdots (j+1+\alpha)x^j.
\end{align} We observe that $b_j=\binom{n}{j}\frac{(\alpha+n)!}{(\alpha+j)!}.$ For $0\le j\le n-p,$ we see that $p|\frac{(\alpha+n)!}{(\alpha+j)!}$ since a product of $p$ consecutive positive integers is divisible by $p.$ Let $n-p <j \leq n-k.$ Then $k \leq n-j <p$ and $\binom{n}{j}=\frac{n(n-1)\cdots(j+1)}{(n-j)!}.$ Therefore $p|\binom{n}{j}$ since $p|n(n-1)\cdots(n-k+1)$. Hence $p|b_j$ for $0 \leq j \leq n-k.$ Therefore it suffices to show that 
 \begin{align}\label{Delfa}
 \frac{\nu_p(\Delta_j)-\nu_p(\binom{n}{j})}{j}<\frac{1}{k} \ \ {\rm for} \ 1\le j\le n.
 \end{align}
 Clearly this is true for $j=1, 2, \cdots, k$ by our assumption. Hence we take $j>k$.  Since
 $\Delta_j=\frac{(\alpha+j)!}{\alpha!}$, we have by Lemma \ref{order}
  \begin{align*}
 \frac{\nu_p\left(\Delta_j\right)}{j}=\frac{j-\sigma(\alpha+j)+\sigma(\alpha)}{(p-1)j}=\frac{1}{p-1}+
 \frac{\sigma(\alpha)-\sigma(\alpha+j)}{(p-1)j} \ \ {\rm for} \ 1\le j\le n.
 \end{align*}
 Let $j\ge \alpha$. Since $p\ge 2k+1, \sigma(\alpha)\le \alpha$ and $\sigma(\alpha+j)\ge 1$, we have
 \begin{align*}
 \frac{\nu_p\left(\Delta_j\right)}{j}\leq  \frac{1}{2k}+\frac{\alpha-1}{2kj}\leq \frac{1}{2k}+\frac{\alpha-1}{2k\alpha}<\frac{1}{k}.
 \end{align*}
 Hence we may suppose that $j<\alpha \leq 50$. Then $\alpha+j\leq 2\alpha-1\le 99$.  Let $1\leq j_0\leq j$ be such that
 $\max_{1 \leq i \leq j}\nu_p(\alpha+i)=\nu_p(\alpha+j_0):=\nu_0$. Then
 \begin{align*}\begin{split}
 \frac{\nu_p\left(\Delta_j\right)-\nu_p(\binom{n}{j})}{j} & \leq \frac{\nu_p(\alpha+j_0)+\nu((j-1)!)-\nu_p(\binom{n}{j})}{j} \\
 & \le \frac{\nu_0+\frac{j-2}{p-1}-\nu_p(\binom{n}{j})}{j}
 \end{split}
 \end{align*}
 using Lemma \ref{order}. Since $p\ge 2k+1$, $\frac{j-2}{(p-1)j}\leq \frac{j-2}{2kj}=\frac{1}{2k}-\frac{1}{kj}$ and therefore
 \begin{align*}
 \frac{\nu_p\left(\Delta_j\right)-\nu_p(\binom{n}{j})}{j}  \le \frac{\nu_0-\frac{1}{k}-\nu_p(\binom{n}{j})}{j}+\frac{1}{2k}.
 \end{align*} Then $\frac{\nu_p(\Delta_j)-\nu_p\left(\binom{n}{j}\right)}{j}<\frac{1}{k}$ if $j>2\nu_0k-2-2k\nu_p(\binom{n}{j})$. Hence we now suppose
 $j\le 2\nu_0k-2-2k\nu_p(\binom{n}{j})$.
 
 Let $k=2$. Then $p\ge 5.$ Since $\alpha+j\le 99$, we have $\nu_0\le 2$. Hence $j\leq 6-4\nu_p(\binom{n}{j})$.
 Further $j\ge 3$ since $j >k$. Hence $3\le j\le 6-4\nu_p(\binom{n}{j})$ implying $j\le 6$ and  $\nu_p(\binom{n}{j})=0$.
 Since $p|\binom{n}{j}$ for $2\le j<p$, we have $p=5$ and $j\in \{5, 6\}$. Further we have from $\alpha\leq 50$ that 
 $\frac{\nu_5\left(\Delta_5\right)}{5}\leq \frac{2}{5}<\frac{1}{2}$ giving \eqref{Delfa}. 
 Hence we need to consider only $j=6$ and it suffices to show that
 \begin{align*}
 \frac{\nu_5(\Delta_6)-\nu_5\left(\binom{n}{6}\right)}{6}<\frac{1}{2}.
 \end{align*}
 If $5\nmid (\alpha+1)$, then $\nu_5(\Delta_j)\le 2$ and we are done. Hence $5|(\alpha+1)$. Since
 $\nu_p(\alpha+1)\le \nu_p(n)$ by our assumption, we have $5|n$ and further
 \begin{align*}
 \frac{\nu_5(\Delta_6)-\nu_5\left(\binom{n}{6}\right)}{6}&=\frac{\nu_5(\alpha+1)+\nu_5(\alpha+6)-\nu_5(n)-\nu_5(n-5)+1}{6}\\
 &\leq \frac{\nu_5(\alpha+6))-\nu_5(n-5)+1}{6}\leq \frac{\nu_5(\alpha+6))}{6}\leq \frac{2}{6}<\frac{1}{2}
 \end{align*}
 since $5|(n-5)$ and $\alpha+6\le 56$. 
 
 Let $k=1$. Then $j\leq 2\nu_0-2-2\nu_p(\binom{n}{j})$. Let $\nu_0\le 2$. Then $j\le 2-2\nu_p(\binom{n}{j})$.
 Recall that $j\ge 2$ and hence $j=2$ and $\nu_p(\binom{n}{2})=0$ which is not possible since $p\ge 3, p|n$. Thus
 $\nu_0\ge 3$. Then $p=3$. Further from $\alpha+j\leq 99$ we get $\nu_0\le 4$ implying $j\leq 6$ which together with 
 $\alpha\leq 50$ gives $\nu_0=3$ and hence  $j\le 4-2\nu_p(\binom{n}{j})$. In particular, $j\le 4$. When $j=4$, we have $3|\binom{n}{4}$
 and $\nu_3(\Delta_4)\le 4$ and hence the assertion \eqref{Delfa} is valid. Thus $j\in \{2, 3\}$. This, together with $\nu_0=3$ 
 and $j_0\le j,$ implies $\alpha+j_0=27$ and $\nu_3(\alpha+j_0)=3$. Further  $\nu_3(\Delta_j)=\nu_3(\alpha+j_0)$ since $j \in \{2,3\}.$ Therefore we may assume $3-\nu_3(\binom{n}{j}) \ge j$ if $j=2, 3$ 
 else \eqref{Delfa} is valid. Let $j=2.$ Then $\nu_3(\binom{n}{2})=\nu_3(n) \geq 1$ implying $\nu_3(n)=1.$ We see that $j_0 \neq 1,$ otherwise $3=\nu_3(\alpha+1) \leq \nu_3(n)$ by our assumption and this is not possible. Thus $j_0=2$ implying $\alpha=25.$ Let $j=3.$ Then $\nu_3(\binom{n}{3})=\nu_3(n)-1.$ Hence we can assume that $\nu_3(n)=1$ and $j_0 \neq 1$ as in the above case. This give $\alpha \in \{24,25\}$ and $\nu_3(n)=1.$ 
 \end{proof}
 
 Let
 \begin{align*}
 S_M=\{n\geq 1 : n , P(n(n+1)) \leq M \}.
 \end{align*}
 The sets $S_M$ for $M\le 41$ are given by tables in Lehmer \cite[Table IIA]{Leh64} and for  $M=100$ by table in Luca and Najman \cite{Naj} and \cite{Naj1}.
  \begin{lemma}\label{alpha >40}
   Let $k=2$ and $40 < \alpha \leq 50.$
   Then \ps \ has no factor of degree $2.$
  \end{lemma}
  
  \begin{proof}
   Assume that \ps \ has  a factor of degree $2.$ If $P(n(n-1)(n+\alpha)(n+\alpha-1)) \geq 53,$ then  \ps \ has no factor of degree $2$ by Lemma \ref{Glemma}. Hence we may assume that $P(n(n-1)(n+\alpha)(n+\alpha-1)) \leq 47.$ We refer to the tables of \cite{Naj} to find $n$ and $\alpha$ such that $P(n(n-1)(n+\alpha)(n+\alpha-1)) \leq 47.$ For these pairs $(n,\alpha),$ we find a prime $p$ to apply Lemmas \ref{Glemma} and \ref{Glemma1} to conclude that \ps \ has no factor of degree $2$ except for pairs $(n, \alpha) \in \{(4,45),(4,46),(6,44),(8,41),(9,41),(12,43),(16,48)\}.$ Let $(n,\alpha)=(12,43).$ Then $NP_3(\g)$ where $\g$ is given by \eqref{g} has vertices  
   \begin{align*}
   \{(0,0),(9,5),(12,7)\}.
   \end{align*}
   We derive the different possibilities for  $NP_3(\psi_n^{(\alpha)}(x))$ using $NP_3(\g).$ If $3|a_3,$ then the vertices of $NP_3(\psi_n^{(\alpha)}(x))$ are given by $\{(0,0),(12,7)\}.$ Hence by Lemma \ref{Dumas}, \ps \ has no factor of degree $2.$ If $3 \nmid a_3,$ then the vertices of $NP_3(\psi_n^{(\alpha)}(x))$ is same as $NP_3(\g).$ Again by Lemma \ref{Dumas}, \ps \ has no factor of degree $2.$ Now we apply Lemma \ref{fil} with the following choice of primes for each of the other values of $n$ and $\alpha$ to conclude that \ps \ has no factor of degree $2.$ 
   \begin{center}
   \begin{tabular}{|c|c|}
    \hline
    $p$ & $(n,\alpha)$\\
    \hline
    2&(4,45),(16,48)\\
    \hline
     3  & (4,46) ,(9,41)\\
    \hline 
     7 & (6,44),(8,41) \\
    \hline
        \end{tabular}
     \end{center}
 \end{proof}
 
Denote by $T$ the set of all triplets $(n,\alpha,k)$ listed in Lemma \ref{from extensions}. Further we put\\
$T_1:=$ \{(8,13,2),(6,19,2),(9,19,2),(8,20,2),(4,21,2), (12,21,2),(24,22,2),(16,24,2),(9,27,2),\\(18,33,2),(16,34,2),(9,40,2),(27,38,2),(14,12,3),(16,12,4)\}. We observe that $T_1$ is a subset of $T.$
 \begin{lemma}\label{k >1}
 Let $2 \leq k \leq \frac{n}{2}$ and $11\leq  \alpha \leq 50.$
 Then \ps \ has no factor of degree $k$ except for $(n,\alpha,k) =(16,24,2).$
  \end{lemma}
  
 \begin{proof}
 Assume that $\psi_n^{(\alpha)}(x)$ has  a factor of degree $k.$ By Lemma \ref{alpha >40}, we may assume that $\alpha \leq 40$ when $k=2.$ Since the irreducibility of $f_n^{(\alpha)}(x)$ implies the irreducibilty of \ps, we may assume, by Lemma \ref{from extensions}, that \ps \ has no factor of degree $k \geq 2$ except for the triplets $(n,\alpha,k) \in T.$ We consider $(n,\alpha,k)=(7,2,2) \in T.$ Here $p=7$ divides $n$ but does not divide $(\alpha+1)(\alpha+2)=12$ and $p >2k$ and $p^2-p \geq \alpha.$ Now we derive from Lemma \ref{Glemma} that \ps\ has no factor of degree $2.$ For $(n,\alpha,k)=(6,3,2) \in T,$ we are not able to find a prime $p$ satisfying Lemma \ref{Glemma}, but we apply Lemma \ref{Glemma1} with $p=5$ to conclude that \ps \ has no factor of degree $2.$ We apply Lemmas \ref{Glemma} and \ref{Glemma1} similarly to conclude that we are left with $(n,\alpha,k) \in T_1$ among triplets in $T$.  Let $(n,\alpha,k)\in \{(14,12,3), (16,12,4)\}$. When 
 $(n,\alpha,k)=(14,12,3),$ we have $NP_7(g_n^{(\alpha)}(x))=\{(0,0),(14,2)\}$ and when $(n,\alpha,k)=(16,12,4),$  we have $NP_2(g_n^{(\alpha)}(x))=\{(0,0),(16,15)\}$. Therefore these cases are excluded by Lemmas \ref{fil} and \ref{Dumas}, respectively.  
 We now take $(n,\alpha,k) \in T_1-\{(16,24,2)\}$ and may suppose that $k=2$. We calculate the Newton polygons for $\g$ given by 
 \eqref{g} in each of these cases with a suitable prime so that the conditions of Lemma \ref{fil} are satisfied. Then we calculate the slope 
 of the right most edge in each case. If the slope of the right most edge is $<\frac{1}{2}$ , we exclude it by Lemma \ref{fil} and the cases 
 where the slope of the right most edge is $\geq \frac{1}{2}$  are excluded by applying Lemma \ref{Dumas}.  We illustrate this by some 
 examples. Let $(n ,\alpha)=(6,19).$  Then the vertices for $NP_3(\g)$ are given by $\{(0,0),(6,2)\}$ and the slope of the right most edge is $\frac{1}{3} <\frac{1}{2}$. Hence \ps \ does not have a factor of degree $2$ by  Lemma \ref{fil}. Let $(n,\alpha)=(9,19).$ The vertices for $NP_3(\g)$ are given by $\{(0,0),(9,5)\}.$ Here $NP_3(\g)$ is same as $NP_3$(\ps) and the maximum slope is $\frac{5}{9}>\frac{1}{2}$. 
 However $NP_3$(\ps) has only one edge with lattice points $(0, 0)$ and $(9, 5)$.  Hence \ps \ has no factor of degree $2$ by Lemma \ref{Dumas}.  
  
 For each of the following pairs of $(n,\alpha)$, we give a choice of a prime $p$ for considering its Newton polygon and then we conclude as above that \ps\ has no factor of degree $2$. 
  \begin{center}
\begin{tabular}{|c|c|}
 \hline
  $p$ & $(n,\alpha)$\\
 \hline
 2& (8,20),(12,21),(16,34)\\
 \hline
  3 & (4,21),(9,40),(18,33),(27,38)\\
 \hline 
 5& (9,27),(24,22)\\
 \hline
 \end{tabular}
  \end{center}
  Now it remains to consider the pair $(n,\alpha)=(8,13).$ We calculate $NP_7(\g)=\{(0,0),(7,1),(8,2)\}$. We consider the possibilities for $NP_7(\psi_n^{(\alpha)}(x)).$ If $7|a_1,$ then $NP_7(\psi_n^{(\alpha)}(x))=\{(0,0),(8,2)\}$ and if $7 \nmid a_1,$ then $NP_7(\psi_n^{(\alpha)}(x))=\{(0,0),(7,1),(8,2)\}.$ In both cases it is clear from Lemma \ref{Dumas} that \ps \ has no factor of degree $2.$ 
 \end{proof}

 Next we formulate a computational lemma.
 \begin{lemma}\label{computational}
 Let $11\leq  \alpha \leq 50$ and $2 \leq n \leq 50.$ Then \ps \ has no linear factor except for $(n,\alpha) \in \Omega$ with $n \leq 50.$
\end{lemma}
\begin{proof}  Assume that \ps\ has a linear factor and $(n,\alpha) \notin \Omega$. First we consider $\alpha \in \{24,25\}$ and $\nu_3(n)=1$. Then $n \in \{3,6,12,15,21,24,30,33,39,42,48\}.$ We exclude the pairs given by $\alpha=24, n \in \{15,21,33,39,42\}$ and $\alpha=25, n \in \{3,6,12,21,24,30,33,42,48\}$ by Lemma \ref{Glemma}. Let $(n,\alpha)=(3,24).$ We may assume that \ps \ $=x^3+81a_1x^2+2106a_2x\pm 17550$.  Let $x-b$ be a linear factor for \ps. 
Then $b\mid 17550$. Since $17550=2 \cdot 3^3 \cdot 5^2 \cdot 13,$ $b$ is composed of primes $\{2,3,5,13\}$ and $5^3 \nmid b.$ We consider the polynomial
$g_n^{(\alpha)}(x)=x^3+81x^2+2106x+17550.$ Then 
$NP_2(g_n^{(\alpha)}(x))  = \{(0,0),(1,0),(3,1)\},$ 
$NP_3(g_n^{(\alpha)}(x))  = \{(0,0),(3,3)\},$
$NP_5(g_n^{(\alpha)}(x))  = \{(0,0),(1,0),(2,0),(3,2)\},$
$NP_{13}(g_n^{(\alpha)}(x)) = \{(0,0),(1,0),(3,1)\}.$
Since the slope of the right most edge of $NP_p(\psi_n^{(\alpha)}(x))$ is at most equal to that of $NP_p(g_n^{(\alpha)}(x))$, we see that the slope of the right most edge of $NP_2(\psi_n^{(\alpha)}(x))$ and that of $NP_{13}(\psi_n^{(\alpha)}(x)) <1.$ Thus $2\nmid b$, and $13 \nmid b$ by Lemma \ref{monic}. Further $NP_3(\psi_n^{(\alpha)}(x))=NP_3(g_n^{(\alpha)}(x))$ and hence by Lemmas \ref{Dumas} and \ref{monic}, we have $3\parallel b.$  Write $b=3b_1$ with 
$b_1\in \{\pm  1, \pm 5, \pm 5^2\}$. Then $\psi_n^{(\alpha)}(b)=0$ implies $3^3\{b^3_1+3^3a_1\pm 3^2\cdot 26a_2\pm 650\}=0$. This gives  
$3^2|(b^3_1\pm 650)$ which is not true for $b_1\in \{\pm  1, \pm 5, \pm 5^2\}$.  Thus \ps \ has no linear factor when $(n,\alpha)=(3,24).$ Now consider 
the pairs given by $\alpha=24,n=6$ and $\alpha=25, n \in \{15,39\}$ since the remaining pairs are in $\Omega.$ These pairs are excluded by 
Lemma \ref{fil} with $p=2,(n,\alpha)=(6,24); p=5,(n,\alpha)=(15,25)$ and $p=13,(n,\alpha)=(39,25).$   We may now assume that either 
$\alpha \in \{24, 25\}, \nu_3(n) \neq 1$ or $\alpha \notin \{24,25\}.$ All these pairs other than $52$ pairs are excluded by  Lemmas \ref{Glemma} or 
\ref{Glemma1} and the $52$ pairs are excluded by Lemmas \ref{Dumas} and \ref{fil} as in the proof of Lemma \ref{k >1}. 
\end{proof}

\begin{lemma}\label{linear factors}
For $(n,\alpha) \in \Omega,$
 there exists $a_0,a_1,\dots,a_n \in \mathbb{Z}$ with $|a_0|=|a_n|=1$ such that \ps $=\psi_{\alpha}(x;a_0,a_1,\dots,a_n)$ has a linear factor.
  \end{lemma}
 \begin{proof}
 Let $(n,\alpha)=(40,24) \in \Omega.$ Assume that $x-b$ is a factor of \ps. Then $b$ divides the constant term  of \ps \ given by $25 \cdot 26 \cdot 27 \cdots 64. $ Let $p \geq 7$ be a prime dividing $25 \cdot 26 \cdot 27 \cdots 64$. We find that the slope of right most edge of $NP_p(\g) < 1.$ Since the slope of the right most edge of $NP_p$(\ps) \ is at most equal to  that of $NP_p(\g),$ we see the the slope of the right most edge of $NP_p$(\ps)\ $<1.$ Thus $p \nmid b$ by Lemma \ref{monic}. For $p \leq 5,$ the details of vertices of $NP_p(\g)$ are given below.
 \begin{align*}
 NP_2(\g) &= \{(0,0),(32,32),(40,41)\}\\
 NP_3(\g) &= \{(0,0),(1,0),(10,4),(37,17),(40,20)\}\\
  NP_5(\g) &= \{(0,0),(10,2),(35,8),(39,9),(40,10)\}.
 \end{align*}
In each of the above cases, $NP_p(\g)$ has lattice points which give edges of length $1$ and slope $1.$ Here we consider $x+30$ as a possible linear factor. Equating the remainder obtained by dividing \ps\ with $x+30$ to be $0$ and solving the equation in integers, we get the values for $a_j'$s. If $x+30$ is a factor of \ps\, then we observe that $x-30$ is a factor of $-\psi_n^{(\alpha)}(-x).$
 The details of the linear factors for other pairs $(n,\alpha) \in \Omega$  are given in Table 1. 
  \end{proof}
  
   \begin{lemma}\label{alpha=24}
   Let $P(n) \geq 3,\nu_3(n)=1, n >50$ and $\alpha \in \{24,25\}.$ Then \ps\  has no linear factor except when $(n,\alpha)=(120,24)$.
   \end{lemma}
   \begin{proof}
  Let $\alpha=24$ and $\nu_3(n)=1.$ By Lemma \ref{Glemma}, we may assume that 
    \begin{align}\label{condition}
    {\rm if}\  p|n(n+\alpha), {\rm then \ either}\ p|(\alpha+1) \ {\rm or} \ p^2-p < \alpha.  
    \end{align}
    Thus $P(n(n+24)) \leq 5.$ By Lemma \ref{Glemma1}, we may assume that $\nu_5(n) \leq 1.$ Thus we have 
    \begin{align}\label{neweq}
    n&=2^{\alpha_1} 3\cdot 5^{\gamma_1}, \quad \gamma_1 \leq 1\\
    n+24&=2^{\alpha_2} 3^{\beta_2} 5^{\gamma_2}, \quad \beta_2 \geq 1   
    \end{align}
    where $\alpha_1,\alpha_2,\beta_2,\gamma_1,\gamma_2$ are non-negative integers. Let $\gamma_1=0.$ Then $n=2^{\alpha_1}\cdot 3 >50$ implies $\alpha_1 \geq 5.$ Thus $\alpha_2=$ord$_2(n+24)=3$. Thus the above equations give 
    \begin{align}\label{new}
        3^{\beta_2-1} 5^{\gamma_2}-2^{\alpha_1-3}=1.
        \end{align}
    We solve this equation using Lemma \ref{weger}. From now onwards, we solve the diophantine equation $x+y=z$ with $x \leq y,\ P(xyz) \leq 13$ and $\gcd(x,y)=1$ by using Lemma \ref{weger} without reference. Therefore $\alpha_1-3 \leq 12,\beta_2-1 \leq 7,\gamma_2 \leq 5.$ Further the table mentioned in Lemma \ref{weger} does not give any solution to \eqref{new}. Thus using the above bounds for $\alpha_1,\beta_2$ and $\gamma_2$ in \eqref{new}, we get $n\in \{3,6,12,24,30,120,1920\}$ and this is a contradiction as $n >50$ and $\gamma_1=0.$ 
    Therefore we can assume that $\gamma_1=1$ and consequently $\gamma_2=0.$ Further $\alpha_1 \geq 2$ since $n >50.$ If $\alpha_1 \in \{2,3\}$, then $n \in \{60,120\}$ and assume that $\alpha_1 \geq 4.$ Thus $\alpha_2=$ ord$_2(n+24)=3$ and $n+24=2^3 \cdot 3^{\beta_2} >74$ implying $\beta_2 \geq 3.$ This together with \eqref{neweq} give 
    \begin{align*}
   3^{\beta_2-1} -2^{\alpha_1-3}\cdot 5=1.
    \end{align*}
    We use Lemma \ref{weger} to get $n=1920.$ Thus we have $n \in \{60,120,1920\}$. When $(n,\alpha)=(60,24)$, $NP_7(\g)=\{(0,0),(7,1),(56,9),(60,10)\}$ and when $(n,\alpha)=(1920,24)$, $NP_2(\g)=\{(0,0),(128,127),(384,382),(896,893),(1920,1916)\}$. In both cases the slope of right most edge is $<1$ and by Lemma \ref{fil} we conclude that $\psi_n^{(\alpha)}(x)$  has no linear factor in these cases.
    
    Let $\alpha=25.$ Then by \eqref{condition}, if $p|n(n+25),$ then $p\in \{2,3,5,13\}.$ Further by Lemma \ref{Glemma1}, we may assume that $5\nmid n.$  
    Therefore, since $n>50$, we have 
    \begin{align*}
    n&=2^{\alpha_1} \cdot 3,   \quad  \alpha_1\geq 5\\
    n+25&=13^{\delta_2},  \qquad  \delta_2\geq 2.
    \end{align*}
    By considering above equations modulo $8$, we get $13^{\delta_2}\equiv 1$ modulo $8$ and hence $\delta_2$ is even. Then  
    \begin{align*}
    (13^{\delta_2/2}-5)(13^{\delta_2/2}+5)=2^{\alpha_1} \cdot 3 
    \end{align*}
  is not possible.   
  \end{proof}
   
  \begin{lemma}\label{P(n) >3}
  Let $P(n) \geq 3$ and $11\leq \alpha \leq 50.$ Then \ps\ with $(n,\alpha) \notin \Omega$ has no linear factor.
  \end{lemma}
    \begin{proof}
  Let $p|n$ and $p \geq 3$ and $(n,\alpha) \notin \Omega.$ Then $n >50$ by Lemma \ref{computational}. Further by Lemmas \ref{Glemma1} and \ref{alpha=24}, we may assume that
  \begin{align}\label{order of p in alpha}
  \nu_p(\alpha+1) >\nu_p(n).
  \end{align} Since $\nu_p(n) \geq 1,$ we have $\nu_p(\alpha+1) \geq 2.$ This gives $\alpha+1 \in \{18, 25, 27,36,45, 49,50\}.$ Also \eqref{condition} is valid. 
  Let $\alpha_1,\beta_1,\gamma_1$ and $\alpha_2,\beta_2,\gamma_2$ be non-negative integers. 
  
  Let $\alpha=17.$ Then $P(n(n+17))=3$ by \eqref{condition}. Therefore $n =2^ {\alpha_1}\cdot 3$ by \eqref{order of p in alpha}, $P(n) \geq 3$ and $n+17 =2^ {\alpha_2}.$ Then $\alpha_2 \neq 0$ implying $\alpha_1=0$ which is a contradiction as $n>50.$ 
 
 Let $\alpha=24.$ By \eqref{condition}, \eqref{order of p in alpha} and $P(n) \geq 3$, we have $
  n= 2^{\alpha_1}\cdot 5$ and $n+24= 2^{\alpha_2}.$ Since $n >50, \alpha_1 \geq 4$. Thus $\alpha_2=$ord$_2(n+24)=3$ which is a contradiction since $n>50$.

  Let $\alpha=26$. By \eqref{condition}, we have $P(n(n+26)) \leq 5.$ This together with \eqref{order of p in alpha}, $n >50$ and $P(n) \geq 3$ give $n=2^{\alpha_1}3^{\beta_1}$ with $\beta_1 \in \{1,2\},\alpha_1 \geq 3$ and $
  n+26 =2^{\alpha_2} 5^{\gamma_2}$. Since $\alpha_1 \geq 3,$ we have $\alpha_2=$ord$_2(n+26)=1.$ This gives
  \begin{align*}
  5^{\gamma_2}-2^{(\alpha_1-1)}3^{\beta_1}=13.
  \end{align*}
Let $\alpha_1 \geq 4.$ Consider the above equation modulus $8.$ We have $5\equiv 5^{\gamma_2}(\text {mod} \ 8)$ implying $\gamma_2$ is odd. On the other hand if we consider modulus $3,$ we have $1\equiv 5^{\gamma_2}(\text {mod} \ 3)$ implying $\gamma_2$ is even. Therefore we can assume that $\alpha_1=3$ and this gives $n=72$ since $n>50.$ Then $n+26=98$ and hence 
$P(n(n+26))=7>5.$ which is not possible. 
  
  Let $\alpha=35.$ By \eqref{condition}, we have $P(n(n+35)) \leq 5$. Then by \eqref{order of p in alpha} and $P(n) \geq 3$, we have 
  $ n= 2^{\alpha_1}\cdot 3$ 
  and $ n+35= 2^{\alpha_2}.$ Then $\alpha_1 =0$ and $n=3$ which contradicts $ n>50.$
 
 Let $\alpha=44.$ Then by \eqref{condition}, $P(n) \geq 3$ and by \eqref{order of p in alpha}, we have  $
   n= 2^{\alpha_1}\cdot 3$ with $\alpha_1 \geq 5 $ since $n> 50.$ Further
   $ n+44= 2^{\alpha_2}5^{\gamma_2}7^{\delta_2}$ with $\alpha_2=2.$ Then we have
     \begin{align*}
     11=5^{\gamma_2}7^{\delta_2}-2^{\alpha_1-2}\cdot 3
     \end{align*}
     We check the solutions of this equation by Lemma \ref{weger} and we get $n=96$ since $n >50$. For $(n,\alpha)=(96,44),$ we apply Lemma \ref{fil} with $p=7$ to conclude that \ps \ has no linear factor.
     
      Let $\alpha=48.$ Then by \eqref{condition}, \eqref{order of p in alpha} and $P(n) \geq 3$, we have  $n= 2^{\alpha_1}\cdot 7$ and 
      $n+48= 2^{\alpha_2}5^{\gamma_2}.$ If $\alpha_1=3,$ then $n=56$, $n+48=104$ and $13|(n+48)$ and by \eqref{condition}, this is not possible. If $\alpha_1=4,$ then $n=112$ and $(n,\alpha)=(112,48) \in \Omega$. If $\alpha_1=5,$ then $n=224,n+48=272$ and $17|(n+48)$ and this is not possible by 
      \eqref{condition}. Thus $\alpha_1 \geq 6.$ Then $\alpha_2=4$ and we have 
            \begin{align*}
            3=5^{\gamma_2}-2^{\alpha_1-4}\cdot 7.
            \end{align*}
            Taking congruent modulo $4,$ we conclude that the above equation has no solution.

      Let $ \alpha=49.$ By \eqref{condition}, $P(n) \geq 3$, $n >50$ and \eqref{order of p in alpha}, we have $
       n= 2^{\alpha_1}\cdot 5$ with $\alpha_1 \geq 4$. Thus
        $n+49= 3^{\beta_2}$ and we have 
        \begin{align*}
        49=3^{\beta_2}-2^{\alpha_1}\cdot 5.
        \end{align*}
       By considering the above equation modulo 8, we get $\beta_2$ even. Then
        $$(3^{\beta_2/2} +7)(3^{\beta_2/2}-7) = 2^{\alpha_1}\cdot 5.$$         
     This implies $ 3^{\beta_2/2}-7\in \{2,  4, 10, 20\}$ which does not give solution to the above equation.
  \end{proof} 

\subsection*{Proof of Theorem \ref{Thm1}:} Let $11\leq \alpha \leq 50$ and \ps \ has a factor of degree $1 \leq k \leq \frac{n}{2}.$ By Lemma \ref{k >1}, we may assume that $k=1.$ Let $(n,\alpha) \notin \Omega.$ By Lemmas \ref{computational} and \ref{P(n) >3}, we may assume that $ n >50$ and $P(n)=2.$ Then $n=2^r > \alpha.$ Consider $\g$. The leading coefficient of $\g$ is $ 1$ and its constant term is $(n+\alpha)(n+\alpha-1)\cdots (1+\alpha).$ We apply Lemma \ref{order} to see that
  \begin{align*}
  \nu_2((n+\alpha)(n+\alpha-1)\cdots(1+\alpha))&=\nu_2((n+\alpha)!)-\nu_2(\alpha!)\\&=((n+\alpha)-\sigma(n+\alpha))-(\alpha-\sigma(\alpha))=n-1.
   \end{align*}
    The coefficient of $x^j$ in $\g$ is $\binom{n}{j}(n+\alpha)(n+\alpha-1)\cdots(j+1+\alpha)=\frac{n!}{j!}\binom{n+\alpha}{j+\alpha}$ and 
    \begin{align*}
    \nu_2\left(\frac{n!}{j!}\binom{n+\alpha}{j+\alpha}\right)\geq \nu_2 \left(\frac{n!}{j!}\right)&=\nu_2(n!)-\nu_2(j!)=(n-1)-(j-\sigma(j))\\
    & \geq (n-1)-(j-1)=n-j \ \ \text{for} \ 1 \le j \le n-1.
    \end{align*}
    This implies that $NP_2(\g)$ is the edge joining $(0,0)$ and $(n,n-1)$. Therefore $NP_2(\g)=NP_2(\psi_n^{(\alpha)}(x))$ has only one edge with no 
    lattice point. Thus \ps\ is irreducible and in particular it has no linear factor. Now we apply Lemma \ref{linear factors} to complete the proof of 
    Theorem \ref{Thm1}. \qed

 \section{Proof of Theorem \ref{Thm2}}
 The proof depends on the following result which is analogous to a result for $G_{\alpha}(x)=G_n^{(\alpha)}(x)$ as proved in \cite[Lemma 5.2]{LaNaSh15}. The proof of Lemma \ref{irred} is exactly the same as \cite[Lemma 5.2]{LaNaSh15}.

 \begin{lemma}\label{irred}
 Let $\alpha=u+\frac{1}{2}$, $1\le k\le \frac{n}{2}$ and $a_0,a_1,\dots, a_n \in \mathbb{Z}.$ Suppose there is a prime $p$ with
 \begin{align*}
 p|\displaystyle\prod_{l=0}^{k-1}(1+2u+2(n-l))(n-l), \ \ p\nmid\displaystyle\prod_{l=1}^{k}(1+2u+2l)
 \end{align*}
 satisfying
 \begin{align*}
 p>\max(2k,1+\sqrt{2(u+1)})\  and \ p\nmid a_0a_n.
 \end{align*}
 Then $\psi_n^{(\alpha)}(x^2)$ does not have a factor of degree in
 $\{2k-1, 2k\}$. Further when $n$ is odd and $k=\frac{n-1}{2},\ \psi_n^{(\alpha)}(x^2)$ does not have a factor of degree $n=2k+1.$
 \end{lemma}
 
\subsection*{Proof of Theorem \ref{Thm2}:} Assume that $\psi_n^{(\alpha)}(x^2)$ has a factor of degree $1 \leq l \leq n.$ Recall that $\psi_n^{(\alpha)}(x^2)$ is a special case of $G_n^{(\alpha)}(x^2).$ Let $l \geq 3.$ Then by Lemma \ref{ psi greater than 2 with half}, $(u, n)\in \{(1,12),(6,7),(9, 113),(10,3),(21, 101)\}$  if $l=3,$ or $(u,n) \in S$ if $l=4$ or $(u,n)=(44,79)$ if $l=6$. We apply Lemma \ref{irred} to exclude these possibilities except for $(u,n,l)=(4,9,4)$. Hence we may assume that $l \leq 2.$ Then by Lemma \ref{psi linear half}, we have $(u,n) \in T_0.$ 
 
 Let $(u, n)=(38, 2)$. We may assume that $\psi_n^{(\alpha)}(x^2)=x^4+162ax^2\pm 6399.$ First, we show that $\psi_n^{(\alpha)}(x^2)$ has no linear factor. If 
 not, we get a rational root $r/s$, with $r, s\in \Z$,  $\gcd(r, s)=1, s>0$, of $x^4+162ax^2\pm 6399=x^4+2\cdot 9^2ax^2\pm 9^2\cdot 79$. Hence 
 $r^4+2\cdot 9^2ar^2s^2\pm 9^2\cdot 79s^4=0$ giving $s=1$ and further $r^2|9^2\cdot 79$. Also $9^2|r^4$ and hence  $r^2\in \{3^2, 9^2\}$. We have 
 $r^2\neq 9^2$ else $r^4=9^4|9^2\cdot 79$ which is not possible. Thus $r^2=3^2$ and we obtain $1+18a\pm  79=0$ or 
 $18a\in \{-80, 78\}$ which is not possible. Hence $\psi_n^{(\alpha)}(x^2)$  has no linear factor.  Assume it has an irreducible factor of degree $2$. 
 Then, we can write $x^4+162ax^2\pm 6399= (x^2+A_1x+A_0)(x^2+B_1x+B_0)$ with $A_0,B_,A_1,B_1 \in \mathbb{Z}$. Then 
 $A_0B_0= \pm 6399, A_0B_1+A_1B_0=0=A_1+B_1$ giving $B_1=-A_1$ and hence  $A_0=B_0$ if $A_1\neq 0$.  If $A_1\neq 0$, then $\pm 6399=A_0B_0=A^2_0$  which is not possible. Hence $A_1=0$ giving $B_1=0$ and hence $A_0+B_0=162a$ which together with 
 $A_0B_0= \pm 6399$ gives $(A_0-B_0)^2=(162a)^2 \mp 4\cdot 6399=18^2 \{(9a)^2\mp 79\}$. This  imply $ \pm 79=(9a)^2-y^2=(9a-y)(9a+y)$ for 
 some $y>1$. Then $9a-y=\pm 1$ 
 and $9a+y=\pm 79$ giving $9a=(\pm 1\pm 79)/2$ which is not possible. Thus 
 $\psi_n^{(\alpha)}(x^2)$ is irreducible at $(u, n)=(38, 2)$.  
 
 Let $(u,n) \in T_0-\{(38,2)\}.$ For all these pairs $(u,n)$ we apply Lemma \ref{fil} with suitable primes to conclude that $\psi_n^{(\alpha)}(x^2)$ does not have factor in degree $1.$ Hence we may assume that $\psi_n^{(\alpha)}(x^2)$ has a factor of degree $2$ for all $(u,n) \in T_0-\{(38,2)\}.$ Let $(u,n) \in \{(35,2),(36,2^6)\}.$ We apply Lemma \ref{fil} with $p=3$ for $(u,n)=(35,2)$ and $p=67$ for $(u,n)=(36,2^6)$ to conclude that  $\psi_n^{(\alpha)}(x^2)$ does not have factor in degree $2.$ For $(u,n)=(11,2),$ we find that vertices of $NP_3(g_n^{(\alpha)}(x^2))$ are given by $\{(0,0),(4,3)\}$ and therefore $NP_3(g(x^2))$ is same as $NP_3(\psi_n^{(\alpha)}(x^2)).$ Hence by Lemma \ref{Dumas}, $\psi_n^{(\alpha)}(x^2)$ does not have factor of degree $2$ when $(u,n)=(11,2).$ For all other pairs $(u,n),$ we can always find integers $a_0,a_1,\ldots,a_n$ with $|a_0|=|a_n|=1$ such that  $\psi_n^{(\alpha)}(x^2)$ has a quadratic factor except for $(u,n) =(44,2^{12})$ (see Table 2) by the method described in Lemma \ref{linear factors}. 
 \qed
 
 \subsection*{Proof of Corollary \ref{coro 2}:}  Let $\alpha=u+\frac{1}{2}$ where $u$ is an integer. Suppose $\psi_n^{(\alpha)}(x)$ has a factor of degree $k$. Then  $\psi_n^{(\alpha)}(x^2)$ has a factor of degree $2k.$ Therefore by Theorem \ref{Thm2}, we have $(u,n) \in \Omega_1, k=1$ and the assertion follows from Theorem \ref{Thm2} immediately.
 \qed

\section{Galois Groups: Proof of Theorem \ref{GGLu}}

We will use a result of Hajir \cite{hajir} which gives a criterion for an irreducible polynomial to have large Galois group using Newton polygons. We restate the result which is \cite[Lemma 3.1]{hajir}.

\begin{lemma} \label{An}
Let $f(x)=\sum^m_{j=0}\binom{m}{j}c_jx^j\in \Q[X]$ be an irreducible polynomial of degree $m$. Let $p$ be a
prime with $\frac{m}{2}< p<m-2$ such that
\begin{itemize}
\item[$(i)$] $ord_p(c_j) \geq 0 $ for $j=0,1,\ldots,m,$
\item[$(ii)$] $ord_p(c_0)=1$,
\item[$(iii)$] $ord_p(c_j)\ge 1$ for $0\le j\le m-p$,
\item[$(iv)$] $ord_p(c_p)=0$.
\end{itemize}
Then the Galois group of $f$ contains $A_m.$ Further Galois group is
$A_m$ if disc$(f)\in \Q^{*2}$ and $S_m$ otherwise.
\end{lemma}

 We shall always assume that $\alpha=u+\frac{1}{2}$ in this section where $u$ is an integer $\geq 1$. We define 
 \begin{align*}
 \mathcal{L}_n^{(u)}(x)=\sum^n_{j=0}\binom{n}{j}(1+2(u+n))(1+2(u+n-1))\cdots (1+2(u+j+1))x^j.
 \end{align*} We observe that $ \mathcal{L}_n^{(u)}(2x)=2^nn!L_n^{(\alpha)}(-x)$ and thus the Galois group of $\mathcal{L}_n^{(u)}(x)$ and $L_n^{(\alpha)}(x)$ are same.

We shall be applying the above lemma with $f(x)=\mathcal{L}_n^{(u)}(x)$. In \cite{Sch31}, Schur showed that the discriminant of $\mathcal{L}_n^{(u)}(x)$ is given by
\begin{align*}
D^{(u)}_n:=Disc(\mathcal{L}_n^{(u)}(x))=\prod^n_{j=2}j^j(\frac{2u+1}{2}+j)^{j-1}.
\end{align*}
We write $D^{(u)}_n=bY^2$, $Y\in \Q$ with
\begin{align}\label{b}
b=\begin{cases}
\frac{3\cdot 5\cdots n\cdot (2u+1+4)(2u+1+8)\cdots (2u+1+2(n-1))}{2^\del}& {\rm if} \ n\equiv 1, 3({\rm mod} \ 4)\\
\frac{3\cdot 5\cdots (n-1)\cdot (2u+1+4)(2u+1+8)\cdots (2u+1+2n)}{2^\del}& {\rm if} \ n\equiv 0, 2({\rm mod} \ 4)
\end{cases}
\end{align}
where $\del=0$ if $n\equiv 0, 1($mod $4)$ and $1$ if $n\equiv 2, 3($mod $4)$. Observe that $b$ is never a square when $n\equiv 2, 3($mod $4)$. In the next lemma, we find all pairs $(u, n)$ such that $b$ is a square. 
\begin{lemma}\label{bsq}
Let $u\leq \max(45, \frac{4n}{3})$. The pairs $(u, n)$ for which $b$ given by \eqref{b} is a square are 
$(u, n)=(u,1)$ where $1 \leq u \leq 45$ in which cases $b=1$.  
\end{lemma}
\begin{proof}
We may assume that $n\equiv 0, 1($mod $4)$. Let $\eta=0, 1$ according as 
$n\equiv 0, 1$ modulo $4$, respectively. Let $x=\frac{1+2(u+n-\eta)}{1.048}$ and $x \geq 887.$ Then $n >198$ since $u \leq \max(45,\frac{4n}{3}).$ Further $x > \max(n-1+\eta, 2u+4).$ Therefore by Lemma \ref{prime cong}, the interval  $(\max(n-1+\eta, 2u+4), 1+2(u+n-\eta)]$ 
contains  a prime $p$ congruent to $1+2u$ modulo $4$. Further $2p >1+2(u+n-\eta)$ since $p>x.$ Hence $b$ is not a square. Therefore we may now suppose that 
$x=(1+2(u+n-\eta))/1.048<887$ or $1+2(u+n-\eta)\leq 929$. 

We have $8(u+n)/7\geq 2u$ if $u\leq 4n/3$. Hence taking $m=1+2(u+n-\eta)$, we get 
that $4m/7+5\geq 2u+5-4/7>2u+4$ if $u\leq 4n/3$. Also $4m/7+5>n$. For $m\geq 158$, 
we also have $4m/7\geq 2\cdot 45$. Hence for $158\leq m\leq 929$, we check that 
the interval $(4m/7+5, m]$ contain both primes congruent to $1$ and $3$ modulo 
$4$. Thus we may suppose that $m=1+2(u+n-\eta)\leq 157$. 

Let $2u+4\leq n$. Then $\max(2u+4,n)=n<m/2$. For $7\leq m\leq 157$ and $m$ odd, we check 
that the interval $(m/2, m]$ contain both primes congruent to $1$ and $3$ modulo $4$ 
except for $m=11$. Hence for $1+2(u+n-\eta)\leq 157$, $b$ is not a square except when 
$1+2(u+n-\eta)= 11$ or $1+2(u+n)\leq 6$. These cases can be excluded since $n \geq 2u+4$.

Let $n\leq 2u+3$. Then $3n-2-2\eta \leq  1+2(u+n-\eta)\leq 157$. Thus $n\leq 53$. We check that for 
primes $\leq 157$, gaps between consecutive primes in the same residue modulo $4$ is at 
most 24. Hence from \eqref{b}, we obtain that $b$ is not square if $2(n-1)\geq 24$ or 
$n\geq 13$. Thus we may suppose that $n\leq 12$. Then $n\in \{1, 4, 5, 8, 9, 12\}$.  For 
these values of $n$, we have $u\leq \max(45, 4n/3)=45$ and we check that $b$ is not a square unless $(u,n)=(u,1)$ where $1 \leq u \leq 45$. Hence the assertion.    
\end{proof}

\begin{lemma}\label{pr}
Let $n>1, u \leq \max( 45, \frac{4n}{3})$ and $\mathcal{L}_n^{(u)}(x)$ be an irreducible polynomial. Suppose there exists a 
prime $p$ with $\frac{n}{2}< p<n-2$ such that 
\begin{align}\label{p||}
p||\prod^p_{l=n-p+1}(1+2(u+l)), 
\end{align}
then the Galois group of $\mathcal{L}_n^{(u)}(x)$ is $S_n$. 
\end{lemma}

\begin{proof}
We apply Lemma \ref{An} with $f(x)=\mathcal{L}_n^{(u)}(x).$ Then
\begin{align*}
c_j=(1+2(u+n))(1+2(u+n-1))\cdots (1+2(u+j+1)).
\end{align*}
Since $n/2<p<n-2$, there are at most 2 terms in 
\begin{align*}
1+2(u+1), \ldots, 1+2(u+n)
\end{align*}
divisible by $p$. By \eqref{p||} and $2p-n<p$, there is exactly one $l_p$ 
with $n-p+1\leq l_p\leq p$ and $p||(1+2(u+l_p))$. This together with 
$l_p-p\leq 0$ and $l_p+p>n$ implies $1+2(u+l_p)$ is the only term exactly divisible by $p$ in $1+2(u+1), \ldots, 1+2(u+n)$. Hence $p\nmid c_p$ since 
$l_p\leq p$. Further for $0\leq j\leq n-p$, we have 
ord$_p(c_j)=$ord$_p(1+2(u+l_p))=1$. Hence all the assumptions in Lemma \ref{An} are satisfied. Finally we apply Lemmas \ref{An}, \ref{bsq} and $n >1$ to get the assertion of Lemma \ref{pr}. 
\end{proof}

\begin{lemma}\label{GGu}
Let $u\leq \max(45, \frac{4n}{3})$ and $\alpha=u+\frac{1}{2}$. Suppose $\mathcal{L}_n^{(u)}(x)$ be irreducible. 
Then the Galois group of $\mathcal{L}_n^{(u)}(x)$ is $S_n.$
\end{lemma}

\begin{proof} Let $n \leq 130$ and $u \leq \max(45,\frac{4n}{3}).$ We apply Lemma \ref{pr} for these pairs of $(u,n)$. We check that all these pairs  with $n \geq 40$ are excluded. Out of the remaining we are left with $619$ pairs of $(u,n) $ for which Lemma \ref{pr} is not satisfied. For these 619 pairs, we compute Galois group directly using MAGMA. Hence we may now suppose that $n >130.$ Since $1+2(u+n)<(n/2)^{2}$, by Lemma \ref{pr}, it suffices to choose a prime $p\in (n/2, n-2)$ with $p|(1+2(u+l_p))$ for some $l_p$ such that 
$n-p+1\leq l_p\leq p$. 
 For each 
$p\in (n/2, n-2)$, we write 
\begin{align*}
1+2(u+n)=t_pp+r_p, \ 0\leq r_p<p. 
\end{align*}
It suffices to find a prime $p\in (n/2, n-2)$ such that $p$ divides one of 
$r_p+2, r_p+4, \ldots, r_p+2(2p-n)$. We now restrict to $p\in (2n/3, n-2)$.

Suppose $r_p$ is odd for some $p\in (2n/3, n-2)$. Then $r_p+2(2p-n)\geq p$ 
if $r_p\geq 2n-3p$ which is true since $2n-3p<0\leq r_p$. This, together with 
$r_p<p$ and $r_p$ odd implies $p=r_p+2i$ some $i\leq 2p-n$. Hence we may now 
assume that $r_p$ is even for each $p\in (2n/3, n-2)$. Write $r_p=2r'_p$ with 
$0\leq r'_p\leq (p-1)/2$. Then $r_p+2, r_p+4, \ldots, r_p+2(2p-n)$ is given by 
$2(r'_p+1), 2(r'_p+2), \ldots, 2(r'_p+2p-n)$. If $r'_p+2p-n\geq p$ for some 
$p\in (2n/3, n-2)$, then we are done. Hence assume that $r'_p+2p-n<p$ implying 
$r_p=2r'_p\leq 2n-2p-2$ for each $p\in (2n/3, n-2)$. Further 
\begin{align}\label{tp}
1+2(u+n)<(t_p-2)p+2n \ {\rm and}\  t_p \ {\rm is\ odd \ for \ each}\ p\in (2n/3, n-2)
\end{align}
 since $1+2(u+n)$ is odd. 

We now write $P_1, P_2$ for the least prime and maximum prime in $(2n/3, n-2)$, respectively. 
Then $P_1=p_{\pi([2n/3])+1}>2n/3$ and $P_2=p_{\pi(n-3)}\leq n-3$. We first show the following: 
\begin{align}\label{3P-P}
3P_2-P_1>2n. 
\end{align}
For $130< n\leq 1000$, we check that the above assertion holds. Assume $n>1000$. 
By \cite{pgap}, there is a prime in $(m, 1.05m)$ for every $m\geq 213$. Taking $m=\lfloor 2n/3\rfloor$, we get 
$P_1<1.05\cdot 2n/3=0.7n$. Again taking $m=\lfloor \frac{n-3}{1.05}\rfloor =\lfloor 20(n-3)/21\rfloor$, we get 
$P_2\geq 20(n-3)/21$. Thus $3P_2-P_1\geq 20(n-3)/7-0.7n>2n$. 
 
\noindent 
{\bf Case I:} Let $u$ be such that $1+2(u+n)<3P_2$. Taking $p=P_2$, we get 
$1+2(u+n)=t_pp+r_p<3p$ giving $t_p\in \{0,1, 2\}$ implying $t_p=1$ since $t_p$ is odd. Hence 
$r_p=1+2(u+n)-p>2n-p>2n-2p$ which is a contradiction. Thus  
$1+2(u+n)\geq 3P_2$. 

\noindent 
{\bf Case II:} Let $u$ be such that $3P_2\leq 1+2(u+n)<5P_1$. Taking $p=P_1$, we get 
$3p<3P_2\leq 1+2(u+n)=t_pp+r_p<5p$ giving $t_p=3$ since $t_p$ is odd and $r_p <p$. This gives 
$3P_2\le (3-2)P_1+2n$ or $3P_2-P_1\leq 2n$ which contradicts \eqref{3P-P}. Thus 
$1+2(u+n)\geq 5P_1$. 

\noindent 
{\bf Case III:} Let $u$ be such that $5P_1\leq 1+2(u+n)<5P_2$. Observe that $3P_2<5P_1$ since 
$P_1>2n/3$ and $P_2<n-2$. Taking $p=P_2$, we get $3p<5P_1\leq 1+2(u+n)=t_pp+r_p<5p$ giving $t_p=3$ since $t_p$ is odd and $r_p <p$. Further $5P_1\le (3-2)P_2+2n$. This is a contradiction since 
 $10n/3<5P_1\leq  P_2+2n<3n$. Thus $1+2(u+n)\geq 5P_2$. 

\noindent 
{\bf Case IV:} Let $u$ be such that $5P_2\leq 1+2(u+n)<7P_1$. Taking $p=P_1$, we get 
$5p<5P_2\leq 1+2(u+n)=t_pp+r_p<7p$ giving $t_p=5$. This gives 
$5P_2\le (5-2)P_1+2n$ i.e $3P_2-P_1+2(P_2-P_1)\leq 2n$ which contradicts \eqref{3P-P}. 

Thus $1+2(u+n)\geq 7P_1$. Since $P_1>2n/3$, we have  $P_1\geq 2n/3+1/3$ giving  
$1+2(u+n)\geq 7(2n/3+1/3)$ implying $u>4n/3$. Since $u\leq \max(45, 4n/3)$ and 
$n\geq 130$, this is not possible. Hence the assertion. 
\end{proof}

\noindent 
{\bf Proof of Theorem \ref{GGLu}:} By \cite[Corollary 1.1]{LaNaSh15}, we see 
that $\mathcal{L}_n^{(u)}(x)$ is irreducible except for $(u, n)=(10,3)$. 
For $(u, n)=(10,3)$, we check that the Galois group is $\mathbb{Z}_2$. For 
$(u, n)\neq(10,3)$, the assertion now follows from Lemma \ref{GGu}.
\qed

       \end{document}